\documentclass[11pt]{article}
\usepackage{amsfonts,amsmath,amsthm}

\usepackage[usenames,dvipsnames]{color}
\usepackage[usenames]{color}
\usepackage[colorlinks,%
linkcolor=BrickRed,%
filecolor =BrickRed,%
citecolor=RoyalPurple,%
]{hyperref} 

\def\transpose{{\hbox{\tiny\it T}}}

\newcommand{\RL}{{\mathbb R}}

\newcommand{\IND}{{\mathbb I}}

\def\be{\begin{eqnarray}}
\def\ee{\end{eqnarray}}
\def\ben{\begin{eqnarray*}}
\def\een{\end{eqnarray*}}

\def\sq{$\Box$}

\def\qed{\ifmmode\sq\else{\unskip\nobreak\hfil
\penalty50\hskip1em\null\nobreak\hfil\sq
\parfillskip=0pt\finalhyphendemerits=0\endgraf}\fi\par\medbreak}

\newsavebox{\junk}
\savebox{\junk}[1.6mm]{\hbox{$|\!|\!|$}}

\def\state{{\sf X}}

\newcommand{\field}[1]{\mathbb{#1}}

\def\Re{\field{R}}

\def\ind{\field{I}}

\def\bfPhi{\mbox{\protect\boldmath$\Phi$}}

\def\til={{\widetilde =}}

\def\tilT{\widetilde T}

\def\tiltau{\tilde \tau}

\def\clB{{\cal B}}

\def\clD{{\cal D}}

\def\half{{\mathchoice{\textstyle \frac{1}{2}}%
{\frac{1}{2}}%
{\hbox{\tiny $\frac{1}{2}$}}%
{\hbox{\tiny $\frac{1}{2}$}} }}

\def\eqdef{\mathbin{:=}}

\def\Prob{{\sf P}}

\def\Expect{{\sf E}}

\def\epsy{\varepsilon}

\def\varble{\,\cdot\,}

\newtheorem{theorem}{Theorem}[section]

\newtheorem{proposition}[theorem]{Proposition}
\newtheorem{lemma}[theorem]{Lemma}

\def\Lemma#1{Lemma~\ref{#1}}
\def\Proposition#1{Proposition~\ref{#1}}
\def\Theorem#1{Theorem~\ref{#1}}

\def\Section#1{Section~\ref{#1}}

\def\bdes{\begin{description}}
\def\edes{\end{description}}

\def\FRAC#1#2#3{\genfrac{}{}{}{#1}{#2}{#3}}

\def\half{{\mathchoice{\FRAC{1}{1}{2}}%
{\FRAC{1}{1}{2}}%
{\FRAC{3}{1}{2}}%
{\FRAC{3}{1}{2}}}}

\def\bfmath#1{{\mathchoice{\mbox{\boldmath$#1$}}%
{\mbox{\boldmath$#1$}}%
{\mbox{\boldmath$\scriptstyle#1$}}%
{\mbox{\boldmath$\scriptscriptstyle#1$}}}}

\def\bfmB{\bfmath{B}}

\def\transpose{{\hbox{\tiny\it T}}}

\newcounter{rmnum}
\newenvironment{romannum}{\begin{list}{{\upshape (\roman{rmnum})}}{\usecounter{rmnum}
\setlength{\leftmargin}{24pt}
\setlength{\rightmargin}{16pt}
\setlength{\itemindent}{-1pt}
}}{\end{list}}
 
\newcounter{anum}

\newlength{\noteWidth}
\long\def\notes#1{\ifinner
             {\tiny #1}
             \else
             \marginpar{\parbox[t]{\noteWidth}{\raggedright\tiny #1}}
             \fi}

\def\archive#1{}

\newcounter{tasks}

{\begin{quote}%
\begin{list}{\hspace{-.75cm}\hbox{\textrm{\rm (\roman{tasks})}\ }}{\usecounter{tasks}%
        \setlength{\labelsep}{0pt}
        \setlength{\leftmargin}{0pt}
        \setlength{\rightmargin}{0pt}
        \setlength{\labelwidth}{0pt}
        \setlength{\listparindent}{0pt}}}%
{\end{list}\end{quote}}

\newcounter{tasksA}

{\begin{quote}%
\begin{list}{\hspace{-.75cm}\hbox{\textrm{\rm (\alph{tasksA})}\ }}{\usecounter{tasksA}%
        \setlength{\labelsep}{0pt}
        \setlength{\leftmargin}{0pt}
        \setlength{\rightmargin}{0pt}
        \setlength{\labelwidth}{0pt}
        \setlength{\listparindent}{0pt}}}%
{\end{list}\end{quote}}
 
\begin{document}
 
\title{\vspace{-2cm}%
On the $f$-Norm Ergodicity\\
of Markov Processes in Continuous Time}
\author{
      	I. Kontoyiannis\thanks{Department of Informatics,
		Athens University of Economics and Business,
		Patission 76, Athens 10434, Greece.
                Email: {\tt yiannis@aueb.gr}.
		\newline
		I.K.\ was supported by the European Union and Greek National
	Funds through the Operational Program Education and Lifelong Learning 
	of the National Strategic Reference Framework through the Research 
	Funding Program Thales-Investing in Knowledge Society through the 
	European Social Fund.
		}
\and
        S.P. Meyn\thanks{Department of Electrical and Computer 
                Engineering,
                University of Florida, Gainesville, USA.
                 Email: {\tt meyn@ece.ufl.edu}.
                \newline
                S.P.M. was supported in part by the National Science Foundation
  ECS-0523620,  and AFOSR grant FA9550-09-1-0190.
  Any opinions, findings,
  and conclusions or recommendations expressed in this material are
  those of the authors and do not necessarily reflect the views of the
  National Science Foundation or AFOSR.}
}

\maketitle
\thispagestyle{empty}
\setcounter{page}{0}
 
\begin{abstract}
Consider a  Markov process $\bfPhi=\{ \Phi(t) : t\geq 0\}$ 
evolving on a Polish space $\state$.
A version of the $f$-Norm Ergodic Theorem is obtained:  
Suppose that the process is $\psi$-irreducible and aperiodic. 
For a given function $f\colon\state\to[1,\infty)$,  
under suitable conditions on the process the following are equivalent:
\begin{romannum}
\item There is a unique invariant probability measure $\pi$ 
satisfying $\int f\,d\pi<\infty$.   
\item There is a closed set $C$ satisfying $\psi(C)>0$ 
that is ``self $f$-regular.''
\item There is a function $V\colon\state \to (0,\infty]$ 
that is finite on at least
one point in $\state$, for which
the following Lyapunov drift condition
is satisfied,
\[
\clD V\leq -  f+b\ind_C\, ,
\eqno{\hbox{(V3)}}
\]
where $C$ is a closed small set
and $\clD$ is the extended generator 
of the process.
\end{romannum}
For discrete-time chains the result 
is well-known. Moreover, in that case,
the ergodicity of $\bfPhi$ under a suitable
norm is also obtained:
For each initial condition $x\in\state$ satisfying $V(x)<\infty$,
and any function $g\colon\state\to\Re$ for which $|g|$ is 
bounded by $f$,
\[
\lim_{t\to\infty} \Expect_x[g(\Phi(t))]  = \int g\,d\pi.
\]
Possible approaches are explored for establishing
appropriate versions of corresponding results
in continuous time,
under appropriate assumptions
on the process $\bfPhi$ or on the function $g$.  
 
\bigskip

\bigskip


\bigskip


\end{abstract}

\newpage

\section{Introduction}

Consider a Markov process 
$\bfPhi=\{ \Phi(t) : t\geq 0\}$
in continuous time, evolving on 
a Polish space $\state$, equipped with 
its Borel $\sigma$-field $\clB$.  
Assume it is a nonexplosive Borel right process: 
It satisfies the strong
Markov property and has right-continuous
sample paths  \cite{ethkur86,meytwe93a}.

The distribution of the process $\bfPhi$ is 
described by the initial condition $\Phi(0)=x\in\state$
and the transition semigroup:
For any $t\ge 0$, $x\in\state$,  $A\in \clB$,
\[
P^t(x,A):=\Prob_x\{\Phi(t)\in A\}:=\Pr\{\Phi(t)\in A\,|\,\Phi(0)=x\}.
\]

A set $C$ is called {\em small} if there is probability 
measure $\nu$ on $(\state,\clB)$,  a time $T>0$,  and a 
constant $\epsy>0$ such that,
\[
P^T(x,A) \ge \epsy\nu(A),\qquad \text{\it for every $A\in\clB$.}
\]
It is assumed that the process is $\psi$-irreducible and aperiodic, 
where $\psi$ is a probability measure on $(\state,\clB)$.  
This means that for each set $A\in\clB$ satisfying $\psi(A)>0$,  
and each $x\in\state$,
\[
P^t(x,A)>0,\qquad \text{\it for all $t$ sufficiently large.}
\]
It follows that there is a countable covering of the state 
space by small sets \cite[Prop.~3.4]{meytwe93e}.
 
The Lyapunov theory considered in this paper and in our 
previous work \cite{kontoyiannis-meyn:II,meytwe93a} is based on 
the \textit{extended generator} of $\bfPhi$, denoted $\clD$.   
A function  $h\colon\state\to\Re$ is in the domain of $\clD$ 
if there exists a function $g\colon\state\to\Re$ such that the stochastic process defined by,
\begin{equation}
M(t) =  h(\Phi(t)) - \int_0^t g(\Phi(s))\, ds     ,\qquad t\geq 0,
\label{e:extgenMart}
\end{equation}
is a \textit{local martingale}, for each initial condition $\Phi(0)$ \cite{ethkur86,rogwil00b}.
We then write $g=\clD h$.

For example, consider a diffusion on $\state=\Re^d$,
namely, the solution of the stochastic 
differential equation,
\be
d\Phi(t)=u(\Phi(t))dt +M(\Phi(t))dB(t),
\;\;\;\;t\geq 0,\;\Phi(0)=x,
\label{eq:SDE}
\ee
where $u=(u_1,u_2,\ldots,u_d)^\transpose:\state\to\RL^d$ and 
$M:\RL^d\to\RL^d\times\RL^k$ are  Lipschitz,
and $\bfmB=\{B(t) : t\geq 0\}$ is $k$-dimensional
standard Brownian motion.  If the function
$h\colon\state\to\RL$ is $C^2$ then we can write
\cite{rogwil00b},
\begin{equation*}
\clD h \, (x) = \sum_i  u_i(x) \frac{d}{\, dx_i} h\, (x)
        +
        \half \sum_{ij}
                \Sigma_{ij}(x) \frac{d^2}{\, dx_i \, dx_j} h\, (x) ,
		\qquad x\in\state .
\end{equation*}

The Lyapunov condition considered in this paper is Condition 
(V3) of \cite{meytwe93a}:  
For a function $V\colon\state\to(0,\infty]$ which is
finite for at least one $x\in\state$,  a function 
$f\colon\state\to[1,\infty)$,  a constant $b<\infty$,  
and a closed, small set $C\in\clB$,
$$
\clD V\leq -\delta f+b\ind_C\, .
\eqno{\hbox{(V3)}}
$$
It is entirely analogous to its discrete-time counterpart \cite{MT}, 
in which the extended generator is replaced by a difference operator
$\clD=P-I$, where $P$ is the transition kernel of
the discrete-time chain and $I$ is the identity operator.

The lower bound $f\ge 1$ is imposed in (V3) because this function is used to define two norms:  One on measurable functions $g\colon\state\to\Re$ via,
\[
\| g\|_f \eqdef \sup_{x\in\state}\frac{|g(x)|}{f(x)},
\]
and a second norm on signed measures $\mu$ on $(\state,\clB)$:
\[
\| \mu \|_f = \sup_{g:|g|\le f} | \mu ( g) |.
\]
Our main goal is to establish 
the erodicity of 
$\bfPhi$ 
in terms of this norm:
There is an invariant measure $\pi$ for the semi-group 
$\{P^t\}$ satisfying,
\begin{equation}
\lim_{t\to\infty} \| P^t(x,\varble) - \pi(\varble) \|_f = 0\, .
\label{e:MET}
\end{equation}
The following result is a partial extension of the 
$f$-Norm Ergodic Theorem 
of \cite{MT} 
to the continuous time setting.

\begin{theorem} 
\label{t:ergod-big-f} 
Suppose that the Markov process $\bfPhi$ is $\psi$-irreducible 
and aperiodic, and let
$f\ge 1$ be a  function on $\state$. Then the following
conditions are equivalent:
\begin{description}
\item[(i)]
The semi-group admits an invariant probability measure $\pi$ satisfying:
\[
\pi(f)  \eqdef \int \pi(dx) f(x) <\infty.
\]
\item[(ii)] 
There exists a closed, small set $C \in \clB$ such that,
\begin{equation}
\sup_{x \in C} \Expect_x\Bigl[\int_0^{\tau_C(1)} f(\Phi(t)) 
\, dt \Bigr] <\infty,
\label{e:f-ergo-2}
\end{equation}
where $\tau_C(1) := \inf\{ t\ge 1 : \Phi(t)\in C \}$
and $\Expect_x$ denotes the expectation operator
under $X_0=x$.
\item[(iii)]
There exists a closed, small set $C$ and and an extended-valued non-negative
function $V$ satisfying $V(x_0)<\infty$ for some $x_0\in\state$, 
such that Condition (V3) holds. 
\end{description}
Moreover, if (iii) holds then there exists a constant $b_f$ such that,
\begin{equation}
  \Expect_x\Bigl[\int_0^{\tau_C(1)} f(\Phi(t)) \, dt \Bigr] \le b_f (V(x)+1) ,\quad x\in\state
\label{e:f-reg}
\end{equation}
where $V$ and $C$ satisfy the conditions of (iii).  
The  set $S_V = \{x: V(x) < \infty \}$ is absorbing 
($P^t(x,S_V) =1$ for each $x\in S_V$ and all $t\geq 0$),  
and also full ($\pi(S_V)=1$).
\end{theorem}

\begin{proof}
Theorem~1.2 (b)   of  \cite{meytwe93b} gives the equivalence 
of  (i) and  (ii).  Theorem~4.3  of  \cite{meytwe93b} gives the 
implication (iii) $\Rightarrow$  (ii),
along with the bound \eqref{e:f-reg}.

Conversely,  if (ii) holds then we can define,
\begin{equation}
V(x) = \int_0^\infty \Expect_x\Bigl[ f(\Phi(t)) \exp\Bigl(-\int_0^t \ind\{ \Phi(s) \in C\} \,ds \Bigr)\Bigr] \, dt.
\label{e:V3converse}
\end{equation}
We show in \Proposition{t:V3converse} that this is a solution to (V3)
and that it is uniformly bounded on $C$.
\end{proof}

 
The function $V$ in \eqref{e:V3converse} 
has the following interpretation.  Let $\tilT$ denote an 
exponential random variable that is independent of $\bfPhi$, and denote,
\[
\tiltau_C = \min\Bigl\{ t : \int_0^t \ind\{  \Phi(s)\in C \} \, ds 
=\tilT \Bigr\}.
\]
We then have,
\begin{equation}
V(x) =
\Expect_x\Bigl[\int_0^{\tiltau_C} f(\Phi(t)) \, dt \Bigr],
\label{e:V3converse-tilde}
\end{equation}
where now the expectation is over both
$\bfPhi$ and $\tilT$.
Consequently, this construction is similar to the converse theorems found in 
\cite{MT} for discrete-time models.
 
\Theorem{t:ergod-big-f} is almost identical to the 
$f$-Norm Ergodic Theorem of  \cite{MT},  
except that it leaves out the implications 
to ergodicity of the process.  This brings us to 
two open problems:  Under the conditions of \Theorem{t:ergod-big-f}:
\begin{romannum}
\item[Q1]
\textit{Can we conclude that \eqref{e:MET} 
holds for any initial condition $x\in S_V$?}

\item[Q2]
Assume in addition that $\pi(V)<\infty$.   \textit{Can we conclude 
that there exists a
finite constant $B_f$ such that, for all $x\in  S_V$,}
\begin{equation}
\int_0^\infty \|P^t(x,\varble) - \pi \|_f \, dt  \le B_f (V(x)+1).
\label{e:f-ergo-6}
\end{equation}
\end{romannum}
In discrete time,
questions  Q1 and Q2 are answered in the affirmative by the 
$f$-Norm Ergodic Theorem of  \cite{MT}, with the integral 
replaced by a sum in \eqref{e:f-ergo-6}.   

Q2 is resolved in the affirmative in this paper by an application of the 
discrete-time counterpart:

\begin{theorem} 
\label{t:ergod-Q2} 
Suppose that the Markov process $\bfPhi$ is $\psi$-irreducible 
and aperiodic, and that there is a solution to (V3) with $V$ everywhere finite.   Then there is a constant $B_f^0$ such that for each $x,y\in\state$,
\begin{equation}
\int_0^\infty \|P^t(x,\varble) - P^t(x,\varble) \|_f \, dt  \le B_f^0 (V(x)+V(y)+1) 
\label{e:f-ergo-6xy}
\end{equation}
If in addition $\pi(V)<\infty$, then \eqref{e:f-ergo-6} also holds for some constant $B_f$ and all $x$.
\end{theorem} 

Although the full resolution of Q1
remains open, in Section~\ref{s:ergodic} we discuss how \eqref{e:MET} can be established  under additional conditions on the process $\bfPhi$.

We begin, in the following section, with the proof of the 
implication (ii) $\Rightarrow$ (iii),  which is based on theory of generalized resolvents and $f$-regularity   
\cite{meytwe93e}.   Following this result,
it is shown in \Proposition{t:regEquiv} that 
$f$-regularity of the process is equivalent to 
$f_\Delta$-regularity for the sampled process, 
where $\Delta$ is the sampling interval, and,
\begin{equation}
f_\Delta(x) = \int_0^\Delta \Expect_x[f(\Phi(t))]\, dt,
\;\;\;x\in\state.
\label{e:fDelta}
\end{equation}
This is the basis of the proof of \Theorem{t:ergod-Q2} 
that is contained in \Section{s:ergodic}.

\paragraph{Acknowledgment.} 
The work reported in this note was prompted by a question
of Yuanyuan Liu who, in a private communication, pointed 
out to us that some results in our earlier work
\cite{konmey01a} were stated inaccurately. Specifically:
(1.) 
The implication (ii) $\Rightarrow$ (iii) in
Theorem~2.2 of \cite{kontoyiannis-meyn:I},
which is the same as the corresponding result
in our present Theorem~\ref{t:ergod-big-f},
was stated there without proof;
and (2.) The convergence in 
\eqref{e:MET} was stated as a consequence
of any of the three equivalent conditions
(i)---(iii), again without proof. 
This note attempts to address and correct
these omissions, although the relevant
statements
in \cite{konmey01a} were only discussed
as background material and do not affect
any of the subsequent results in that paper.

\section{$f$-Regularity}
\label{s:fReg}

Following \cite{meytwe93e}, we denote for each $r\ge 0$ and $B\in\clB$, 
\begin{equation}
 G_B(x,f;r)\eqdef  \Expect_x\Bigl[\int_0^{\tau_B(r)} f(\Phi(t)) \, dt \Bigr] ,
\label{e:GCB}
\end{equation}
where $\tau_B(r) = \inf\{ t\ge r : \Phi(t)\in B \}$, and we write $ G_B(x,f) =  G_B(x,f;0)$.
The Markov process is called \textit{$f$-regular} if there exists $r_0>0$ such that  $G_B(x,f;r_0)<\infty$ for every $x$ and every 
$B\in\clB$ satisfying $\psi(B)>0$.

The following result,
given here without proof, is a simple consequence
of Lemma~4.1 and Prop.~4.3 of \cite{meytwe93e}:
\begin{proposition}
\label{t:meytwe93eGCB}
Suppose that the set $C$ is closed and small, and that the following 
self-regularity property holds:  There exists $r_0>0$ such that  
$\sup_{x\in C}G_C(x,f;r_0)<\infty$.    Then:
\begin{romannum}
\item There is $b_C<\infty$ such that $G_C(x,f;r)<G_C(x,f;r_0) + b_C r$ for each $x$ and $r$.
\item For each $B\in\clB$ satisfying $\psi(B)>0$,  for each $r\ge 0$,  and for each $x\in\state $,
\[
G_C(x,f;r) < \infty \Rightarrow G_B(x,f;r) <\infty.
\]
\end{romannum}
Consequently,
the process is $f$-regular if  $G_C(x,f;r_0)<\infty$ for each $x$. 
\end{proposition}
 
We next show that the function $V$ in \eqref{e:V3converse}
is finite-valued on $\{x\in\state: G_C(x,f;r_0)<\infty\}$.  
We show that $V$ is in the domain of the extended generator,  
and obtain an expression for $\clD V$.

Consider the generalized resolvent developed in 
 \cite{meytwe93e,nev72}:  For a function $h\colon\state\to \Re_+$,  $A\in\clB$, and $x\in\state$, denote,
\[
R_h(x,A) =  \int_0^\infty \Expect_x\Bigl[ \ind_A(\Phi(t))  \exp\Bigl(-\int_0^t h(\Phi(s) ) \,ds  \Bigr)\Bigr] \, dt.
\]
With the usual interpretation
of $P^t$, or any kernel $Q(x,dy)$,
as a lineal operator,
$g \mapsto Qg=\int g(y)Q(\cdot,dy)$,
it is shown in \cite{nev72} that the following 
resolvent equation holds:  For any functions $g\ge h \ge 0$,
\begin{equation}
R_h  = R_g + R_g  I_{g-h} R_h,
\label{e:ResolveEqnNeveu}
\end{equation}
where, for any function $g$,
$I_g$ denotes the (operator
induced by the) kernel 
$I_g(x,dy) =g(x)\delta_x(dy)$.

When $h\equiv \alpha$ is constant,
we obtain the usual resolvent,  
\begin{equation}
R_\alpha \eqdef   \int_0^\infty  e^{-\alpha t} P^t\, dt,
\;\;\;\;\alpha>0,
\label{resolvent}
\end{equation} 
In the case $\alpha=1$ we  
write $R:=R_1 = \int_0^\infty  e^{-  t} P^t\, dt$,
and call $R$ ``the'' resolvent kernel.  For any non-negative function $g\colon\state\to \Re_+ $ for which $Rg$ is finite valued, 
the function $\gamma=Rg$ is in the domain of the extended generator, 
with,
\begin{equation}
\clD \gamma = Rg - g.
\label{e:DR}
\end{equation}
 
\begin{proposition}
\label{t:V3converse}
Suppose that the assumptions of  \Theorem{t:ergod-big-f}~$(ii)$ hold:  
There is a closed, small set $C \in \clB$ such that,
$\sup_{x\in C}G_C(x,f;r_0)<\infty$ with $r_0=1$.
Then the function $V$ defined in \eqref{e:V3converse-tilde} 
is finite on the full set $S_V\subset\state$ and (V3) holds 
with this function $V$ and this closed set~$C$. 
\end{proposition}

\begin{proof}
Proposition 4.3 (ii) 
of \cite{meytwe93e}
implies that the set of $x$ for which $G_C(x,f;1) <\infty $ is a full set.
This result combined with Proposition 4.4 (ii) 
of \cite{meytwe93e}
implies that $V$ is bounded on $C$.  

For arbitrary $x$ we have $\tiltau_C>  
\tau_C = \min\{ t\ge 0 : \Phi(t)\in C \}$.   Consequently,  
by the strong Markov property and the representation 
\eqref{e:V3converse-tilde},
\[
\begin{aligned}
V(x) & =
\Expect_x\Bigl[\int_0^{\tau_C} f(\Phi(t)) \, dt \Bigr] 
+
\Expect_x\Bigl[\Expect_{\Phi(\tau_C)} \Bigl[\int_0^{\tiltau_C} f(\Phi(t)) \, dt \Bigr] \Bigr] 
\\
&\le
G_C(x,f;1) +\sup_{x'\in C} V(x').
\end{aligned}
\]
Hence $V(x)$ is finite whenever $G_C(x,f;1) $ is finite.

To establish (V3), first observe that 
the function $V$ in \eqref{e:V3converse-tilde} can be expressed,
\[
V = R_h f,  \qquad \text{with $h=\ind_C$}.
\]
Taking  $g\equiv 1$, the resolvent equation gives,
\[
R_h  = R+ R  I_{1-h} R_h = R[I + I_{C^c} R_h],
\]
where, for any set $B$ and kernel $Q$, $I_BQ$ denotes the
kernel $\IND_B(x)Q(x,dy)$.
Combining the representation of $V$ above with \eqref{e:DR}
we obtain,
\[
\begin{aligned}
V  &= R[I + I_{C^c} R_h]f
\\
\text{and} \qquad
\clD V &=   (R-I) [I + I_{C^c} R_h]f.
\end{aligned}
\]
The second equation can be decomposed as follows,
\[
\clD V  =  D_1 - D_2 -f,
\]
with $D_1=  R [I + I_{C^c} R_h]f  = V$ and  $D_2=I_{C^c} R_h f = I_{C^c} V$.  
Substitution then gives,  
\[
\clD V = -  f +  \ind_C V.
\]
This establishes (V3) with $b=\sup_{x\in C} V(x)$.
\end{proof}

The final results in this section concern the 
\textit{$\Delta$-skeleton chain}.  This is the 
discrete-time Markov chain with transition kernel $P^\Delta$, 
where $\Delta\ge 1$ is given. 
  It can be realized by sampling the Markov process with sampling interval  $\Delta$.  The sampled process is denoted,
\begin{equation}
X(i) = \Phi(i\Delta),\qquad i\ge 0.
\label{e:skeleton}
\end{equation}

In prior work, the skeleton chain is used to 
translate ergodicity results for 
discrete-time Markov chains to the continuous time setting.  
For example,  Theorem 6.1 of \cite{meytwe93a} implies 
that a weak version of the ergodic
convergence \eqref{e:MET} holds for an 
$f$-regular Markov process:
\begin{equation}
\lim_{t\to\infty} \| P^t(x,\varble) - \pi(\varble) \|_1 = 0\, .
\label{e:MET1}
\end{equation}
The proof consists of two ingredients:  
(i) The corresponding ergodicity
result holds for the $\Delta$-skeleton chain, and (ii)
the error $\| P^t(x,\varble) - \pi(\varble) \|_1 $ is non-increasing in $t$.   

In the next section we use a similar approach to address question Q2.   
The $f_\Delta$ norm is considered, where 
the function $f_\Delta$ is defined in \eqref{e:fDelta}.   Denote,
\[
\sigma^\Delta_C =\min\{i \ge 0: X(i)\in C\},\qquad
\tau^\Delta_C =\min\{i \ge 1 : X(i)\in C\}\, .
\]
The
$\Delta$-skeleton is called {\em $f_\Delta$-regular} if,
\[
 G_B^\Delta (x,f_\Delta)\eqdef  
\Expect_x\Bigl[\sum_{i=0}^{\tau^\Delta_B} f_\Delta(X(i)) \Bigr] <\infty,
\]
for every $x\in\state$ and every $B\in\clB$ satisfying $\psi(B)>0$.

\begin{proposition}
\label{t:regEquiv} 
If the process $\bfPhi$ is $f$-regular,
then each $\Delta$-skeleton is $f_\Delta$-regular.  
Moreover,  there is a closed $f$-regular set $C$ such that:
\begin{romannum}
\item For a finite-valued function $V_\Delta\colon\state \to (0,\infty]$   
and a finite constant $b$, 
\begin{equation}
P^\Delta V_\Delta \leq V_\Delta -  f_\Delta+b\ind_C\, ,
\label{e:VDelta3}
\end{equation}
and $\sup_x| V_\Delta(x)
- G_C(x,f) |<\infty$.

\item
For every $x\in\state$ and every $B\in\clB$ satisfying $\psi(B)>0$, 
there is a constant $c_B<\infty$ such that,
\begin{equation}
 G_B^\Delta (x,f_\Delta) \le  G_C(x,f) + c_B.
\label{e:GBbdd}
\end{equation}
\end{romannum}
\end{proposition}

\begin{proof}
It is enough to establish (i).
Theorem 14.2.3 of \cite{MT} then implies that for every $B\in\clB$ satisfying $\psi(B)>0$, there is a constant $c_B^\Delta<\infty$ satisfying $ G_B^\Delta (x,f_\Delta) \le V_\Delta(x) + c_B^\Delta$.    

Let $C$ denote any closed $f$-regular set for the process, satisfying $\psi(C)>0$.  
For $V_0(x) = G_C(x,f)$ we obtain a bound similar to  \eqref{e:VDelta3} through the following steps.  First write,
\[
P^\Delta V_0\, (x) = \Expect_x\Bigl[\int_\Delta^{\tau_C(\Delta)} f(\Phi(t)) \, dt \Bigr] \, .
\]
The integral can be expressed as a sum,
\[
\begin{aligned}
\int_\Delta^{\tau_C(\Delta)} f(\Phi(t)) \, dt  
&= \int_\Delta^{\tau_C(\Delta)} f(\Phi(t)) \, dt \ind\{\tau_C\le \Delta\}
\\
&\quad +\int_\Delta^{\tau_C} f(\Phi(t)) \, dt \ind\{\tau_C > \Delta\}.
\end{aligned}
\]
By the strong Markov property,
\[
\begin{aligned}
  \Expect_x\Bigl[  \ind\{\tau_C\le \Delta\} \int_\Delta^{\tau_C(\Delta)} f(\Phi(t)) \, dt\Bigr] 
&\le 
  \Expect_x\Bigl[ \ind\{\tau_C\le \Delta\} \int_{\tau_C}^{\tau_C(\Delta)} f(\Phi(t)) \, dt \Bigr] 
\\
&\le \Prob_x\{ \tau_C\le \Delta \} \sup_y G_C(y,f;\Delta).
\end{aligned}
\]
Consequently, 
\begin{equation} 
P^\Delta V_0\, (x) \le  \Expect_x\Bigl[\int_\Delta^{\tau_C} f(\Phi(t)) \, dt \Bigr]  +  b_0 s(x)
= V_0(x) - f_\Delta(x) +  b_0 s(x),
\label{e:VDelta3a}
\end{equation}
where $b_0= \sup_y G_C(y,f;\Delta)<\infty$,  and $s(x)=\Prob_x\{ \tau_C\le \Delta \} $.

To eliminate the function $s$ in \eqref{e:VDelta3a} we establish the following bound:  For some  $ \epsy_0>0$ and $k_0\ge 1$,
\begin{equation}
P^{k_0\Delta}(x,C)\ge \epsy_0 s(x),\qquad x\in\state.
\label{e:probsmall}
\end{equation}
The proof is again by the strong Markov property:
\[
\begin{aligned}
P^{k_0\Delta} (x,C)  &\ge \Expect_x[ \ind\{\tau_C\le \Delta \} \ind\{ \Phi(k_0\Delta) \in C\} ]
\\
&=\int_{r=0}^\Delta \int_y  \Prob_x\{ \tau_C\in dr,\, \Phi(r)\in dy\}  P^{k_0\Delta - r} (y,C)
\\
&\ge \epsy(k) s(x),
\end{aligned}
\]
where $ \epsy(k) = \inf \{ P^{k_0\Delta - r} (y,C) : y\in C,\ 0\le r\le \Delta\}$.   This is strictly positive for sufficiently large $k$ because \eqref{e:MET1} holds.
This establishes \eqref{e:probsmall}.

The Lyapunov function can now be specified as,
\[
V_\Delta(x) =  V_0(x)  + b_0G_C^\Delta (x,s),
\]
where $b_0 $ is defined in \eqref{e:VDelta3a}.
The required bound $\sup_x|V_\Delta(x)
- G_C(x,f) |<\infty$ holds because
$V_0(x) =
G_C(x,f) $, and
 the second term is uniformly bounded:
\[
\begin{aligned}
G_C^\Delta (x,s) &= \Expect_x\Bigl[\sum_{i=0}^{\tau^\Delta_B} s(X(i)) \Bigr]
\\
&\le \epsy_0^{-1}   \Expect_x\Bigl[\sum_{i=0}^{\tau^\Delta_B} P^{k_0\Delta} (\Phi(i\Delta),C)   \Bigr]   
\\
&= \epsy_0^{-1}   \Expect_x\Bigl[\sum_{i=0}^{\tau^\Delta_B}  \ind\{X(i+k_0) \in C\}   \Bigr]   
 \le  \epsy_0^{-1} (k_0+1).
\end{aligned}
\]
Consequently, from  familiar arguments,
\[
\begin{aligned}
PV_\Delta(x) - V_\Delta(x) &\le - f_\Delta(x) +  b_0 s(x)
\\
&\quad + b_0 \Bigl\{ G_C^\Delta (x,s) - s(x) + \ind_C(x) \epsy_0^{-1} (k_0+1)  \Bigr\}. \end{aligned}
\]
This establishes \eqref{e:VDelta3}
with $b=b_0
\epsy_0^{-1} (k_0+1) $.
\end{proof}

\def\tiln{\tilde n}
\def\Tcpl{\tilde{\large\tau}}

\section{$f$-Norm Ergodicity}
\label{s:ergodic}
In this section we consider the implications 
to the ergodicity of the process.
We assume that 
(V3) holds for a finite-valued function $
V\colon\state\to (0,\infty)$, so that the process is $f$-regular.

\bigskip


\noindent
{\bf Q1. $f$-norm ergodicity. }
The ergodicity of $\bfPhi$ in terms of the $f$-norm
as in \eqref{e:MET} has only been established under special conditions.   
Theorem~5.3 of \cite{meytwe93b} implies that \eqref{e:MET} will hold if 
$f$ is subject to this additional bound:    For some $\beta\ge 0$,
\[
P^t f \le \beta e^{\beta t} f,\qquad t\ge 0.
\]
This holds for example if $f\equiv 1$ and $\beta=1$.   

It is likely that the application of coupling bounds will 
lead to a more general theory.  
Under stronger conditions on the process,  
such a coupling time was obtained in  \cite{loclou08},
and it was used in \cite{locloulou11} to obtain rates 
of convergence in the law of large numbers.  
However, to construct the coupling time, it is assumed in 
this prior work that the semi-group $\{P^t\}$ admits 
a density for each $t$.  No such assumptions are required 
in the discrete-time setting, so the full answer to Q1 
remains open.  

\bigskip
 
\noindent
{\bf Q2. Proof of Theorem~\Theorem{t:ergod-Q2}. }
The copmplete resolution of Q2 is possible 
by applying  \Proposition{t:regEquiv}, which implies that 
the skeleton chain
$\{X(i) = \Phi(i\Delta) : i\ge 0\}$  is $f_\Delta$-regular.   
The bound \eqref{e:GBbdd} is the main ingredient in the proof 
of \Theorem{t:ergod-Q2},
but we also require the following relationship between a 
norm for the process and a norm for the sampled chain.
\begin{lemma}
\label{t:normEquiv}
For any signed measure $\mu$,
\[
\| \mu \|_{f_\Delta} \ge \int_0^\Delta \| \mu P^t \|_f\, dt,
\] 
where, for any measure $\nu$ and kerner $Q$, 
$\nu Q$ denotes the measure $\nu Q(\cdot)=\int \nu(dx) Q(x,\cdot)$.
\end{lemma}

\begin{proof}
We first consider the right-hand side.  Consider the signed 
measure $\Gamma$ on $[0,\Delta]\times\state$ defined by:
\[
\Gamma(dt, dy) =   \mu P^t (dy)dt .
\]
Define $f^\Delta\colon [0,\Delta]\times\state\to [1,\infty)$ 
via $f(t,y)=f(y)$ for each pair $t,y$, and the associated norm,
\[
\| \Gamma \|_{f^\Delta} = \sup \iint g(t,y) \Gamma(dt, dy),
\]
where the supremum is over all $g$ satisfying $|g(t,y)|\le f^\Delta(t,y)$ 
for all $t,y$.   It is shown next that the norm can be expressed,
\begin{equation}
\| \Gamma \|_{f^\Delta} = 
\int_0^\Delta \| \mu P^t \|_f\, dt.
\label{e:GammaNorm}
\end{equation}

The Jordan decomposition 
theorem \cite{halmos50book}
implies that there is a minimal decomposition, 
$\Gamma = \Gamma_+-\Gamma_-$,
in which the two measures on the right-hand side 
are non-negative,  with disjoint supports denoted $S_+,S_-$,
resoectively.
Hence $|\Gamma | \eqdef \Gamma_+ + \Gamma_-$ is a non-negative measure.
In this notation the norm is expressed,
\[
\begin{aligned}
\| \Gamma \|_{f^\Delta} &=   \iint f^\Delta(t,y) |\Gamma | (dt, dy)   \\
&=
 \iint  f(y) \bigl( \ind_{S_+}(t,y) - \ind_{S_-}(t,y) \bigr) \Gamma(dt, dy) 
\\
  &=  \int_0^\Delta \Bigl[ \int_{y\in\state} f(y) \bigl( \ind_{S_+}(t,y) - \ind_{S_-}(t,y) \bigr) \mu P^t(dy) \Bigr]\, dt.
\end{aligned}
\]
For each $t$, the measure on $(\state,\clB)$ defined by 
$\bigl( \ind_{S_+}(t,y) - \ind_{S_-}(t,y) \bigr) \mu P^t(dy)$ is 
the marginal of $|\Gamma|$, and is hence a non-negative measure 
for a.e.\ $t$.  It follows that for such $t$,
\[
\int_{y\in\state} f(y) \bigl( \ind_{S_+}(t,y) - \ind_{S_-}(t,y) \bigr) \mu P^t(dy) = \| \mu P^t\|_f,
\]
which gives
\eqref{e:GammaNorm}.  

Consider next the left-hand side of the inequality in the lemma.  
Letting $\mu=\mu_+-\mu_-$ denote the Jordan decomposition for 
the signed measure $\mu$,  and 
$|\mu|=\mu_+ + \mu_-$,
we have,
\[
\| \mu \|_{f_\Delta} =  \int f^\Delta (x) |\mu|(dx) = \int_{t=0}^\Delta \int_{x\in\state}  |\mu|(dx) P^t(x,dy) f(y).
\]
The right-hand side can be expressed as,
\[
 \int_0^\Delta \int  |\mu|(dx) P^t(x,dy) f(y) = \iint f^\Delta(t,y) \Lambda_+(dt,dy) +  \iint f^\Delta(t,y) \Lambda_-(dt,dy) ,
\]
where $\Lambda_\pm(dt,dy) =  \mu_\pm P^t (dy)dt$ defines a decomposition:
\[
\Gamma = \Lambda_+-\Lambda_-\,.
\]
It follows that $\| \mu \|_{f_\Delta}  \ge \| \Gamma \|_{f^\Delta} $, 
by the minimality of the Jordan decomposition.  
This bound combined with \eqref{e:GammaNorm}
completes the proof.
\end{proof}

\begin{proof}[Proof of \Theorem{t:ergod-Q2}] 
\Theorem{t:ergod-big-f} combined with
\Proposition{t:regEquiv} establishes $f_\Delta$-regularity of the skeleton chain under (V3):  The skeleton chain satisfies (V3) with Lyapunov function $V_\Delta$ that satisfies 
$\sup_x| V_\Delta(x) - G_C(x,f) |<\infty$.  The  bound \eqref{e:f-reg} in \Theorem{t:ergod-big-f}  implies that $V_\Delta(x) \le b_f^\Delta (V(x)+1) $ for some constant $b_f^\Delta$ and  all $x$.  

Theorem 14.3.4 of \cite{MT} then gives the bound, for some finite constant $M_f^0<\infty$,
\begin{equation}
\sum_{k=0}^\infty \| P^{\Delta k}(x,\varble) - P^{\Delta k}(y,\varble) \|_{f_\Delta}
\le M_f^0(V(x)+V(y)+1).
\label{e:MT14.3.4}
\end{equation}
Next  apply  \Lemma{t:normEquiv} with $\mu(\varble) 
= P^{\Delta k}(x,\varble) - P^{\Delta k}(y,\varble)$ to obtain,
\begin{equation}
\| P^{\Delta k}(x,\varble) - P^{\Delta k}(y,\varble) \|_{f_\Delta}
\ge \int_0^\Delta \| \mu P^t \|_f\, dt,
\label{e:Skeletonf-bdd}
\end{equation}
and recognize that $\mu P^t (\varble)= P^{\Delta k+t}(x,\varble) - P^{\Delta k+t}(y,\varble)$.
Substituting the resulting bound into \eqref{e:MT14.3.4} establishes \eqref{e:f-ergo-6xy}.

The proof of \eqref{e:f-ergo-6} is similar:  If in addition $\pi(V)<\infty$, then 
Theorem 14.3.5 of \cite{MT} gives, for some constant $M_f<\infty$,
\begin{equation}
\sum_{k=0}^\infty \| P^{\Delta k}(x,\varble) -\pi(\varble) \|_{f_\Delta} \le M_f(V(x) +1).
\label{e:MT14.3.5}
\end{equation}
This combined with \eqref{e:Skeletonf-bdd} completes the proof.
\end{proof}
  
\def\cprime{$'$}\def\cprime{$'$}\def\cprime{$'$}

\end{document}